%% file: Inequalitis-eng.tex
\documentclass[a4paper,12pt]{article}
\usepackage[left=1.4cm,right=1.4cm,top=2cm,bottom=2cm,bindingoffset=0cm]{geometry}
\usepackage[cp1251]{inputenc}
\usepackage{amsmath}
\usepackage{amsthm}
\usepackage{amssymb}
\usepackage{tikz}
\usepackage{graphicx}
\usepackage{wrapfig}

\usetikzlibrary{patterns}

\newtheorem{theorem}{Theorem}
\newtheorem{lemma}{Lemma}

\begin{document}
\author{{\bf  E.\,V.~Shchepin \thanks{Steklov Math. Institute, scepin@mi-ras.ru} ,
 E. Yu. Mychka \thanks{Lomonosov Moscow State University, mychkaevg@mail.ru }}}

\title{On lower estimations of square-linear ratio for plane Peano curves.}

\date{February, 2021}

%
 \maketitle

\maketitle

\begin{abstract}
It is proved that for any mapping of a unit segment to a unit square, there is a pair of points of the segment for which the
square of the Euclidean distance between their images exceeds the distance between them on the segment by at least  $3\frac58$ times.
And the additional condition that the images of the beginning and end of the segment belong to opposite sides of the square increases the estimate to $4+\varepsilon$.
 \end{abstract}

 In the article \cite{Sch2004} the problem of finding the smallest number of $\kappa$ for which there
exists a mapping $p(t)$ of a unit segment to a unit square such that for any pair of distinct points
the inequality is satisfied for the segment
 \begin{equation}\label{sq-lin1}
   \frac{|p(t_1)-p(t_2)|^2}{t_1-t_2}\le \kappa
 \end{equation}
The numerator of the fraction in this inequality is the square of the Euclidean distance between the points $p (t_i)$,
and the fraction itself is called the square-linear ratio of a pair of points on the curve.
In the same paper, it was proved that $\kappa\ge 3.5$. Essentially the same result was somewhat earlier
obtained in the paper \cite{Nie} in discrete form, that is, concerned with the linear ordering of two-dimensional lattices.
In the last paper, also in discrete form, it is essentially proved that $\kappa\le4$. In continuous form,
the proof of the latter inequality can be found, for example, in the paper \cite{Sch2020}.

The square-linear ratio (abbreviated SLR) is defined for a plane curve $p (t)$
as the supremum of the square-linear ratios of all possible pairs of its points and is denoted by $\kappa (p)$.
The main result of this paper is to prove the inequality $\kappa(p)\ge3\frac{5}{8}$ for
any curve $p(t)$ that maps a segment to a square, whose area is numerically equal to the length of the segment.
In addition, we also consider the question of the lower SLR estimate for maps with fixed origin and
the end.
It is proved that if the ends of the curve described above also belong to opposite sides of the square,
then her SLR will be strictly greater than 4.
 This allows us to improve the result from \cite{Hav2} (Theorem 4),
also giving a lower bound of 4 for fractal maps, but in the form of a non-strict inequality.

\begin{lemma}\label{chain-var}Let $f\colon [a,b]\to\mathbb R^2$ be a continuous mapping of the segment to the plane. And let us give the sequence $t_0<t_1<\dots<t_{n-1}<t_n$ of the points of the segment. Then the inequality is satisfied
\begin{equation}\label{chain-var1}
\max_{i\le n} \frac {|f(t_i)-f(t_{i-1})|^2}{t_{i}-t_{i-1}}\ge \frac {\sum\limits_{i=1}^{n}|f(t_i)-f(t_{i-1})|^2}{t_n-t_{0}},
\end{equation}
 which turns into equality only if, for any $i$, the equality is satisfied
\begin{equation}\label{chain-var2}
\max_{i\le n}  \frac{|f(t_i)-f(t_{i-1})|^2}{t_{i}-t_{i-1}}= \frac{|f(t_i)-f(t_{i-1})|^2}{t_{i}-t_{i-1}}
\end{equation}
\end{lemma}

\begin{proof}
Left to the reader.
\end{proof}

\begin{lemma}\label{circle} Let $f\colon [a,b]\to\mathbb R^2$ is a continuous mapping of the line segment in the plane for which the maximum square-linear ratio is achieved on a pair of end points $p(a) p(b)$. Then the image $f(I)$ is contained in the circle, the diameter of which is a pair of images of the endpoints of the interval, and for every $t$, with the image of $p(t)$ that belong to the boundary of a circle, a square-linear ratios of pairs $p(a),p(t)$ and $p(t),p(b)$ be the maximum.
\end{lemma}
\begin{proof}
Indeed, for any $t\in(a,b)$, by virtue of the lemma \ref{chain-var}, there is an inequality
\begin{equation}\label{circle-ineq}
 |f(t)-f(a)|^2+|f(b)-f(t)|^2\le (b-a)\max\left \{ \frac{|f(t)-f(a)|^2}{t-a},\frac{|f(b)-f(t)|^2}{b-t}\right \}\le  |f(b)-f(a)|^2,
\end{equation}
 that shows that $f (t)$ lies in the specified circle.
If instead of the inequality we have equality, then
by the same Lemma, we conclude maximal condition on the ratios $\frac{|f(t)-f(a)|^2}{t-a}$
$\frac{|f(b)-f(t)|^2}{b-t}$.
\end{proof}

\begin{lemma}\label{antipodes} Let $f \colon I\to C$ be a continuous map of a segment on a convex two-dimensional plane set that is centrally symmetric with respect to zero. Then there is a point $x$ belonging to the boundary of the set $C$, such that the convex hull of its preimage $f^{-1} (x)$ contains some of the preimages of its antipode $f^{-1}(-x)$.
\end{lemma}

\
\input{fig1}

\begin{proof}
Assuming the opposite, we come to the conclusion that the boundary points are divided into two types. The first --- in which all the points of the preimage precede the points of the preimage of the antipode, and the second --- in which follow them. But it is quite clear that the set of points of the first type is open. The same is true for the second type. But then these sets, due to the connectivity of the segment $I$, must intersect. This gives the desired contradiction.
\end{proof}

For the curve $f\colon I\to I^2$ and the point of the image $V$,
we denote by $V_f$ the first moment of the curve passing through $V$, and by $V^f$ --- the last such moment.
Thus we have the following relations:
\begin{itemize} \label{moments}
\item $V_f\le V^f$
\item $f(V_f)=V=f(V^f)$
\item $f^{-1}(V)\subseteq [V_f,V^f]$
\end{itemize}

\begin{lemma}\label{circulate}If unclosed curve $f\colon [0,1]\to [ABCD]$ maps the unit interval to the unit square
with vertices $ABCD$, so that $A_f<B_f<C_f<D_f$, and has a square-linear ratio $\kappa(f)\le 4$, then one has
\begin{enumerate}
\item $|AB|=|BC|=|CD|=1$
\item $A^f<B_f<B^f<C_f<C^f<D_f$
\item there is a point $E$ on the  side $AD$, such that $E_f< B_f$ and $E^f>C^f$.
\end{enumerate}
\end{lemma}

\begin{proof}
  We prove the first statement.
If we assume that $AB$ is the diagonal of the square, then so is $CD$. In this case, the sum of the squares of the links of the chain $ABCD$
is $2+1+2$, which leads to a violation of the condition $\kappa (f)\le 4$.
Similarly, it is shown that the diagonal cannot be $CD$.
Suppose now that the diagonal is $BC$. In this case, the sum of the squares of the lengths of the $ABCD$ chain is 4.
And the  Lemma \ref{chain-var}
allows us to state that the square-linear ratios of the pairs $A_f, B_f)$, $B_f, C_f$, $C_f, D_f$ are all maximal and
equal to four, and in addition $A_f=0$, $D_f=1$.
According to the Lemma  \ref{circle}, the image of the segment $[C_f, 1]$ is contained in a semicircle with $ CD$ as a diameter. Since this semicircle
does not contain any neighborhood of the vertex $D$ in the square $ABCD$, then any neighborhood of this vertex
contains points belonging to the image of the segment $[0, C_f]$.
And the closeness of this image allows us to say that in this case, it contains the point $D$ itself, which contradicts the choice of $D_f$
as the smallest element of the preimage $f^{-1}(D)$. Thus, the validity of the first statement of the lemma is established.

To proof  the second statement suppose the opposite.
Suppose, for example, that $D_f> A^f>B_f$. In this case, we get increasing chains
$A_f<B_f<A^f<C_f<D_f$ or $A_f<B_f<C_f<A^f<D_f$, with the sum of the squares of the distances between their images equal to five.
If we assume that $D_f<A^f$, then we get $A^f=1$ and $A_f=0$, that is, that the curve $f$ is closed, contrary to the conditions of the lemma.
If we assume that $B^f>C_f$, we get the chains $A_f<B_f<C_f<B^f<D_f$ or $A_f<B_f<C_f<D_f<B^f$ with the corresponding  sum of the squares of the distances equal to five.
And the inequality, $C^f<D_f$ it is justified in exactly the same way as the inequality $A^f<B_f$.

Let us proceed to the proof of the third statement of the lemma.
By virtue of the  lemma \ref{antipodes}, there are points $E$ and $F$ that belong to the boundary of the square and are centrally symmetric,
that $E_f<F_f<E^f$.
Suppose that $E$ lies on the side of $AB$.
If we assume that $E_f>B_f$, then for $E^f<C_f$ we have a chain $ABEFECD$ on the curve with the sum of squares greater than four,
and for $E^f>C_f$, we have chains $ABCED$ or $ABCDE$, both with the sum of squares greater than four.
Hence, if the square-linear ratio of the curve does not exceed four,
then $E_f$ must lie between $A_f$ and $B_f$. But in this case, $F_f$ no longer lies in this interval,
because otherwise the chain $AFBCD$ would give the sum of squares more than four. If we assume that $E^f<C_f$, then the sum of squares is greater than four
gives the chain $ABFECD$. All possible chains of $AEB(FC) (ED)$, obtained by rearranging the letters inside the brackets in places, give the sum of squares greater than four. Thus, in all cases, the assumption that $E$ lies on $AB$ leads us to find a chain with a sum of squares of more than four. To the same conclusion, even shorter, we are led by the assumption that $E$ belongs to $BC$ or $CD$.
So the only remaining possibility for $E$ --- is on the $AD$side.

We will show that $E_f<B_f$. The opposite assumption gives a chain $A_f, B_f, E_f, F_f, E^f$ with the sum of the squares of the distances between images greater than four.
And the inequality $E^f>C^f$ is justified by the appearance of the chain $E_f, B_f, E^f, C^f, D_f$ otherwise.
Thus, $E$ is the desired point of the curve.
\end{proof}

\begin{theorem}
  \label{dif-side} If mapping a unit segment to a unit square translates the ends of the segment to opposite sides,
then there is a pair of points whose square-linear ratio is greater than four.
\end{theorem}

\
\input{fig2}

\begin{proof}
Suppose, on the contrary, that the curve $p(t)$ satisfies the conditions of the lemma and has a SLR of at most four.
In this case, based on the lemma \ref{circle}, we denote the vertices of the square of the image in the order of their passage through the
curve with the letters $ABCD$, and we have
$AB=BC=CD=AD=1$.
In addition, it is clear that the beginning of the curve is on the side of $AB$, and the end --- on $CD$.
According to the  Lemma \ref{circle}, $AD$ has a point $O$ traversed twice. First time before $B$, second time --- after $C$.
If we denote by $S$ the starting point of the curve, and by $E$ --- the end point, then depending on the order of passage
points $O$ we get four options for passing the seven listed points SOABCDOE, SAOBCDOE, SOABCODE, SAOBCODE

Since $SO^2+O A^2+AB^2=SA^2+AO^2+OB^2$ and $CO^2+OD^2+DE^2=CD^2+DO^2+OE^2$, so far in all embodiments, the sum of squares
lengths of chain to be expressed in the following way
 \begin{equation}\label{SOE}
SA^2+2x^2+1+1+1+2(1-x)^2+ED^2,
\end{equation} which
takes its minimum value of 4 only when
\begin{equation}\label{condit}
S=A, E=D, x=\frac12
\end{equation}
In this case, the square-linear ratio for all links of the chain is the maximum possible and is equal to 4 by virtue of the lemma
 \ref{chain-var}.
So the \eqref{condit} conditions are met anyway.

We now prove that the vertices A and D are traversed twice, namely, the point A must be traversed again after O,
and the point $D$ --- before the point $O$, for all the described traversal options.

Consider the $AOB$ movement option. Since the square-linear ratio of the pair $AO$ is maximal, then according to the Lemma \ref{circle}, the
image of the segment $[A_p, O_p]$ is contained in a circle with a diameter of AO. Since this circle does not contain any
neighborhood of the vertex A in the square ABCD, the points arbitrarily close to A belong to the image $p[O_p, B_p]$.
And the closeness of this image implies the existence of such a $ \alpha\in [O_p, B_p]$ that $p (\alpha)=A$.
Similarly, we justify the existence of $ \delta\in [C_p, O^p]$, for which $p (\delta)=D$.

By virtue of the  lemma \ref{circle} on pairs of curve points corresponding to the arguments $O_p, A_p$, $A_p,B_p$,
$O^p, D_p$ and $C_p, D_p$ the square-linear ratio is maximal, i.e. equal to four.
Since $$[O_p, B_p]=[O_p,A_p]]\cup [A_p,B_p],$$ in so far as, by virtue of the lemma \ref{circle}, we
obtain that the image of the segment $[O_p, B_p]$ lies in the union of two circles, the first of which has a diameter of the side
the square $AB$, and the second --- segment $AO$. The last circle also contains the image $[A_p, O_p]$.
Similar reasoning shows that the segment of the curve $[C_p, D_p]$ is contained in the union of two circles: the first
has a diameter of the side $CD$, and the second --- the segment $OD$. And as cut curve $[B_p,C_p]$ is contained in the circle
with a diameter of $CD$, we get that the whole curve must lie in the Union of five circles with centres on the sides of the square ABCD:
the three circles having the diameters of the side $AB$, $BC$ and $CD$, and two with diameters $AO$ and $OD$. But the totality of these
the circle does not cover the entire square. For example, no neighborhood of the point $O$ is covered. Thus, we have obtained a contradiction that
proves the theorem.
\end{proof}

Similar arguments to the above allow us to briefly prove that the SLR of any mapping of a unit segment to a
unit square is strictly greater than $3.5$. The following theorem allows us to refine this result.
\begin{theorem}If the continuous map $f \colon I\to I^2 $ maps a unit segment to a unit square, then there is a pair of points $x, y\in I$ such that
$$|f(x)-f(y)|^2> 3\frac58 |x-y|.$$
\end{theorem}

\
\input{fig3}

\begin{proof}
By virtue of the  lemma\ref{circle}, in the future we will assume that $E$ belongs to the side of $AD$. Let the point $E$ be shifted from the center of the side of the square to the left by $x$. The described situation is shown in the figure.

If we assume that $E_f>A_f$, then there are four variants of the chains $A(BE) C(DE)$. For each of them, the sum $\sigma $ of squares of the lengths of the chain links is estimated from below as:

$ABECDE:$ $\sigma\geq4$;

$ABECED:$ $\sigma\geq4$;

$AEBCDE:$ $\sigma=3x^2-x+3\frac34\geq3\frac23$;

$AEBCED:$ $\sigma\geq4$.

\noindent So we are left to consider the case where $E_f<A_f$. In the same way, the analysis of cases with a possible location of the moment $E^f$ leaves only the case $E^f>D_f$ to consider.

The sets $f[A_f,D_f]$ and $f[0,A_f]\cup f[D_f,1]$ are closed and cover the entire square and, in particular, cover the segment $\{0.5\}\times [0,1]$. The first of them does not contain the middle of the bottom side of the square, and the mid --- upper. Otherwise, by adding this midpoint in the chain $EABCDE$ between the corresponding letters, we will get an increase in the sum of the squares of the chain by at least $\frac12$, that is, up to four. Therefore, there is a point $G$ on this segment that belongs to the intersection of these sets. Let the point $G$ be located at the distance $y$ from the lower side of the square.

The point $G$ belongs to $f[A_f, D_f]\cap f[0,A_f]$ or $f[A_f, D_f]\cap f[D_f, 1]$. Let's analyze the case when $G$ belongs to $f[A_f, D_f]\cap f[0, A_f]$. The other case is treated similarly.

There are two possible cases: $G_f<E_f$ and $G_f>E_f$.

1. Consider the case of $G_f<E_f$, which in turn is divided into three sub-cases, depending on which vertices of the square $G$ was traversed a second time.

1a. Let's analyze the sub-case when $G^f$ is between $A_f$ and $B_f$. Let's estimate the sum of squares of the $GEAGBCDE $ chain from below. The sum  $\sigma$ is numerically expressed in the following quadratic form:
$$(x^2+y^2)+(1/2-x)^2+(1/4+y^2)+(1/4+(1-y)^2)+2+(1/2+x)^2,$$ whose
minimum is reached at $x=0$ and $y= \frac13$ and is equal to $3\frac{2}{3}$.

1b. Let's analyze the sub-case when $G^f$ is between $B_f$ and $C_f$. Let's estimate the sum of squares of the chain $GEABGCDE$ from below. The sum  $\sigma$ is numerically expressed in the following quadratic form:
$$(x^2+y^2)+(1/2-x)^2+1+(1/4+(1-y)^2)+(1/4+(1-y)^2)+1+(1/2+x)^2,$$ whose
minimum is reached at $x=0$ and $y= \frac23$ and is equal to $3\frac{2}{3}$.

1b. Let's analyze the sub-case when $G^f$ is between $C_f$ and $D_f$. Let's estimate the sum of squares of the chain $GEABCGDE$ from below. The sum  $\sigma$ is numerically expressed in the following quadratic form:
$$(x^2+y^2)+(1/2-x)^2+2+(1/4+(1-y)^2)+(1/4+y^2)+(1/2+x)^2,$$
the minimum of which is reached at $x=0$ and $y= \frac13$ and is equal to $3\frac{2}{3}$.

2. Consider the case $G_f>E_f$, which in turn is divided into three sub-cases, depending on which vertices of the square $G$ was traversed a second time.

2a. Let's analyze the sub-case when $G^f$ is between $A_f$ and $B_f$. Let's estimate the sum of squares of the chain $EGAGBCDE$ from below. The sum  $\sigma$ is numerically expressed in the following quadratic form:
$$(x^2+y^2)+2 (1/4+y^2)+(1/4+(1-y)^2)+2+(1/2+x)^2,$$ whose
minimum is reached at $x=-\frac14$ and $y=\frac14$ and is equal to $3\frac{5}{8}$.

2b. Let's analyze the sub-case when $G^f$ is between $B_f$ and $C_f$. Let's estimate the sum of squares of the chain $EGABGCDE$ from below. The sum of $\sigma$ is numerically expressed in the following quadratic form:
$$(x^2+y^2)+(1/4+y^2)+1+2(1/4+(1-y)^2)+1+(1/2+x)^2,$$ whose
minimum is reached at $x=-\frac14$ and $y=\frac12$ and is equal to $3\frac{7}{8}$.

2b. Let's analyze the sub-case when $G^f$ is between $C_f$ and $D_f$. Let's estimate the sum of squares of the chain $EGABCGDE$ from below. The sum $\sigma$ is numerically expressed in the following quadratic form:
$$(x^2+y^2)+(1/4+y^2)+2+(1/4+(1-y)^2)+(1/4+y^2)+(1/2+x)^2,$$
the minimum of which is reached at $x=-\frac14$ and $y=\frac14$ and is equal to $3\frac{5}{8}$.

\end{proof}

\end{document}

%% file: fig1.tex
\begin{wrapfigure}[12]{r}{130pt}
\begin{tikzpicture}[scale=2,>=latex]

\filldraw[red](0,-1.5) node[black,scale=1]{figure 1};

\draw[line width=1pt] (-1,1) -- (1,1);
\draw[line width=1pt] (1,-1) -- (1,1);
\draw[line width=1pt] (1,-1) -- (-1,-1);
\draw[line width=1pt] (-1,1) -- (-1,-1);

\filldraw(1,1) circle (1pt) node[above right = 0.3,black,scale=0.9] {C};
\filldraw(-1,1) circle (1pt) node[above left = 0.3,black,scale=0.9] {B};
\filldraw(1,-1) circle (1pt) node[below right = 0.3,black,scale=0.9] {D};
\filldraw(-1,-1) circle (1pt) node[below left = 0.3,black,scale=0.9] {A};

\draw[line width=1pt,dashed] (-0.2,-1) -- (0.2,1);
\filldraw(-0.2,-1) circle (1pt) node[below  = 0.5,black,scale=0.9] {E};
\filldraw(0.2,1) circle (1pt) node[above = 0.5, black,scale=0.9] {F};


\end{tikzpicture}


\end{wrapfigure}

%% file: fig2.tex
\newcount\s
\begin{wrapfigure}[13]{r}{120pt}
\s=2
\begin{tikzpicture}[scale=\s]

\filldraw[red](1,-0.5) node[black,scale=\s*0.5]{figure 2};

 \draw[line width=0.5pt](0,0)--(2,0);
\draw[line width=0.5pt](2,0)--(2,2);
 \draw[line width=0.5pt](2,2)--(0,2);
\draw[line width=0.5pt](0,2)--(0,0);

 \draw[line width=0.5pt,dash pattern=on 1.5pt off 3pt] (1.98,2) arc [start angle=0,end angle=-176,radius=0.98];
 \draw[line width=0.5pt,dash pattern=on 1.5pt off 3pt] (1.98,1.98) arc [start angle=90,end angle=270,radius=0.98];
 \draw[line width=0.5pt,dash pattern=on 1.5pt off 3pt] (0,1.98) arc [start angle=90,end angle=-90,radius=0.98];

 \draw[line width=0.5pt,dash pattern=on 1.5pt off 2.5pt] (1.98,0) arc [start angle=0,end angle=181,radius=0.5];

 \draw[line width=0.5pt,dash pattern=on 1.5pt off 2.5pt] (1,0) arc [start angle=0,end angle=181,radius=0.5];

\filldraw[](0,2.2)node[scale=\s*0.5]{B};
 \filldraw[](2,2.2)node[scale=\s*0.5]{C};

 \filldraw[](0,-0.2)node[scale=\s*0.5]{A};
 \filldraw[](2,-0.2)node[scale=\s*0.5]{D};
 \filldraw[](1,-0.2)node[scale=\s*0.5]{O};

\end{tikzpicture}
\end{wrapfigure}

%% file: fig3.tex
\begin{wrapfigure}[13]{r}{130pt}
\begin{tikzpicture}[scale=2,>=latex]

\filldraw[red](0,-1.5) node[black,scale=1]{figure 3};

\draw[line width=1pt] (-1,1) -- (1,1);
\draw[line width=1pt] (1,-1) -- (1,1);
\draw[line width=1pt] (1,-1) -- (-1,-1);
\draw[line width=1pt] (-1,1) -- (-1,-1);

\filldraw(1,1) circle (1pt) node[above right = 0.3,black,scale=0.9] {C};
\filldraw(-1,1) circle (1pt) node[above left = 0.3,black,scale=0.9] {B};
\filldraw(1,-1) circle (1pt) node[below right = 0.3,black,scale=0.9] {D};
\filldraw(-1,-1) circle (1pt) node[below left = 0.3,black,scale=0.9] {A};

\draw[line width=1pt,dashed] (-0.2,-1) -- (0.2,1);
\filldraw(-0.2,-1) circle (1pt) node[below  = 0.5,black,scale=0.9] {E};
\filldraw(0.2,1) circle (1pt) node[above = 0.5, black,scale=0.9] {F};

\draw[line width=1pt,dashed] (0,-1) -- (0, 1);
\filldraw(0,-0.6) circle (1pt) node[right = 0.5, black,scale=0.9] {G};


\end{tikzpicture}


\end{wrapfigure}